\documentclass[11pt]{article}
\usepackage{amsmath, amssymb, amsfonts, amsthm, tikz,enumerate}
\usepackage{graphicx}

\theoremstyle{plain}
\newtheorem{lemma}{Lemma}[section]
\newtheorem{theorem}[lemma]{Theorem}

\newtheorem{proposition}[lemma]{Proposition}
\newtheorem{prop}[lemma]{Proposition}

\newtheorem*{theorem*}{Theorem}

\newtheorem*{shred}{Shredding Lemma}
\newtheorem*{conjecture}{Conjecture}

\newtheorem{definition}[lemma]{Definition}
\theoremstyle{remark}
\newtheorem{remark}[lemma]{Remark}

\def\RR{\mathbb{R}}
\def\NN{\mathbb{N}}
\def\R{\RR}

\def\eps{\epsilon}
\def\CB{\mathcal{B}}
\def\cR{\mathcal{R}}
\def\cC{\mathcal{C}}
\def\cD{\mathcal{D}}
\def\cU{\mathcal{U}}

\def\cB{\mathcal{B}}
\def\cW{\mathcal{W}}
\def\cO{\mathcal{O}}
\def\cA{\mathcal{A}}
\def\cM{\mathcal{M}}
\def\cI{\mathcal{I}}
\def\cJ{\mathcal{J}}

\def\mfm{\mathcal{M}_f(M)}
\def\mm{\mathcal{M}(M)}

\def\diam{\operatorname{diam}}
\def\interior{\operatorname{Int}}
\def\M{\mathcal{M}}
\def\Hom{\operatorname{Homeo}}

\def\id{\operatorname{id}}

\title{Ergodic theory of generic continuous maps}
\author{Flavio Abdenur \footnote{Partially supported by a CNPq/Brazil productivity-in-research (PQ) grant and by a FAPERJ/"Junior Rio de Janeiro State Scientist" fellowship}, Martin Andersson\footnote{With financial support from Fondation Sciences Math\`ematiques \`a Paris (France) and FAPERJ (Brazil).} }

\begin{document}

\maketitle

\begin{abstract}

We study the ergodic properties of generic continuous dynamical systems on compact manifolds. As a main result we prove that generic homeomorphisms have convergent Birkhoff averages under continuous observables at Lebesgue almost every point. In spite of this, when the underlying manifold has dimension greater than one, generic homeomorphisms have no physical measures --- a somewhat strange result which stands in sharp contrast to current trends in generic differentiable dynamics. Similar results hold for generic continuous maps. 

To further explore the mysterious behaviour of $C^0$ generic dynamics, we also study the ergodic properties of continuous maps which are conjugated to expanding circle maps. In this context, generic maps have \emph{divergent} Birkhoff averages along orbits starting from Lebesgue almost every point. 

\bigskip


\medskip


\end{abstract}

\section{Introduction}

One of the best-known results in ergodic theory, due to Ulam-Oxtoby \cite{MR0005803}, which in fact gave birth to Baire Category arguments in dynamics, states that generic volume preserving homeomorphisms on a compact manifold are ergodic. Although seven decades have passed, a dissipative (i.e., non-volume preserving) analogue of their theorem has still not appeared. This is not because of a lack of interest in the dynamics of generic homeomorphisms, which is in fact a very active area of research in dynamical systems, treated extensively in the works of Alpern and Prasad \cite{MR1826331} and Akin et. al. \cite{MR1980335}. While Alpern and Prasad consider the ergodic theory of generic volume preserving homeomorphisms, Akin et. al. study topological properties in the generic dissipative case. We blend the two approaches and consider ergodic properties of generic dissipative homeomorphisms. In order to do that, we must first decide what we consider to be the appropriate questions to ask. In the differentiable setting it has long been more or less clear what these should be: one should ask whether a generic diffeomorphism has some (possibly many) \emph{physical measures} capturing the statistical behaviour of most orbits. We have found that applying the same type of questions to generic homeomorphisms leads to fashinating insights.

Physical measures have been much in vogue ever since they were introduced by Sinai, Ruelle, and Bowen in  the 70's and shown  to exist for every $C^2$ Axiom A diffeomorphism \cite{MR0415683}. No robust obstacle to the existence of physical measures is known in differentiable dynamics, which is quite generally believed to be a $C^r$ dense phenomenon \cite{MR1755446}. Some doubt, however, has been expressed by Ruelle \cite{MR1858471}. He seriously considers the possibility of some robust mechanism that provides \emph{historical behaviour} --- his term for the lack of convergence of Birkhoff averages on a set of positive Lebesgue measure (a phenomenon beautifully illustrated in a famous example due to Bowen). 

The current work fills the vacuum left after the result of Ulam-Oxtoby by proving that, in the context of generic homeomorphisms, no historical behaviour exists. 

\begin{theorem*}
For a generic homeomorphism $f$ of any compact manifold $M$, the Birkhoff averages
\[ \frac{1}{n} \sum_{k=0}^{n-1} \varphi(f^k(x)) \]
of every continuous $\varphi: M \rightarrow \mathbb{R}$ are convergent Lebesgue almost everywhere.
\end{theorem*}

We also reveal a surprising cenario, very different from that expected in the differentiable setting: although Birkhoff averages exist Lebesgue almost everywhere, a generic homeomorphism has no physical measure (except for the very special case where the underlying manifold is the circle). The lack of physical measures in this context is not at all due to historical behaviour, but to an entirely different phenomenon, unconceivable in the context of generic differentiable dynamics. It is the same phenomenon that appears in the differentiable world in very rigid situations such as the identity map and rational translations of the torus: Birkoff averages exist, but any two points do (Lebesgue) almost surely have different ones. 

Most of the research behind this paper took place at PUC-Rio, ENS de Paris, Université Paris 13, and UFF-Niter\'oi. We express our gratitude to our anonymous referees whose advice has greatly simplified some of our proofs, and to Artur Avila for some discussion on pizza slice dynamics.

\section{Some background}

As mentioned in the introduction, the theory of $C^0$-generic dynamics has been extensively studied from the topological viewpoint,  mainly by M. Hurley and extended in the book \cite{MR1980335}. It is known that generic continuous maps (and in particular generic homeomorphisms) of sufficiently regular topological spaces -- say compact manifolds -- have highly complicated and even pathological dynamics from the topological point of view. For instance, for generic homeomorphisms (i) the nonwandering set is a Cantor set \cite{MR1980335} which contains a dense subset of periodic points \cite{MR1307531}, (ii) any Baire-residual point in the manifold has an adding machine as both its omega- and alpha-limit sets \cite{MR1980335}, and (iii) there are infinitely many topologically repelling sets and infinitely many topologically attracting sets, infinitely nested within each other: every attractor contains repellors, and vice-versa \cite{MR1424400}.

In the introduction to their delightful book, Hurley et al ask whether their results admit ergodic analogues: ``the question of whether something analogous to our results can be obtained in the measure theoretic category is an open one'' (page 1). They later discuss how to approach this question, and suggest the use of Lebesgue measure: ``while there are difficulties in finding an appropriate measure on the space of homeomorphisms, at least on a manifold Lebesgue measure is certainly appropriate in the context of questions on the behavior of ``most points'''' (page 5).

In this paper we follow their suggestion: instead of looking at topologically dynamical properties of a Baire-residual set of points, we examine ergodic-theoretic properties of a full-Lebesgue-measure set of points. We show in various contexts that the dynamics of generic continuous maps are indeed also very pathological (\emph{weird} or \emph{wacky} or, sometimes, even \emph{wicked} -- see Definition \ref{definitions_weird} below) when viewed from this ergodic perspective.

\section{The Results}

Before we state our results we shall provide some vocabulary and notation. We begin with notation for the spaces we work in and an explanation of  what is meant by ``generic''. 

Throughout this paper $M$ denotes a compact connected boundaryless smooth manifold.\footnote{See remark \ref{remark on Lebesgue}.} We denote by $C^0(M)$ the space of all continuous maps from $M$ to itself and by 
$\Hom(M)$ the space of all homeomorphisms of $M$ to itself. Both of these spaces are endowed with the usual $C^0$ metric
\begin{equation}
d_{C^0}(f,g) = \sup_{x\in M} d(f(x),g(x)).
\end{equation}

In so doing, the space $C^0(M)$ becomes a complete metric space whereas $\Hom(M)$ does not. However, there is another metric on $\Hom(M)$, defined by
\begin{equation}
d_{\text{Hom}}(f,g) = d_{C^0}(f,g) + d_{C^0}(f^{-1},g^{-1}).
\end{equation} 

This metric is complete and generates the same topology as $d_{C^0}$ does --- the $C^0$ topology. In practice we shall only use the metric $d_{C^0}$, simply denoted by $d$.

\smallskip

A subset $\cR$ of a topological space $X$ is \emph{residual} if it contain the intersection $\bigcap_{k \in \NN} V_k$ of a countable family of open-and-dense subsets $V_k$ of $X$. A topological space $X$ is a \emph{Baire space} if every residual subset of $X$ is dense in $X$. By the Baire Category Theorem, every complete metric space is Baire. In particular the spaces $C^0(M)$ and $\Hom(M)$ are Baire with respect to the $C^0$ topology.

\smallskip
\begin{definition}
A property $\mathcal{P}$ is said to be \emph{generic} in the space $X$ if there exists a residual subset $\cR$ of $X$ such that every element $p \in \cR$ satisfies property $\mathcal{P}$.
\end{definition}
 Note that, given a countable family of generic properties $\mathcal{P}_1, \mathcal{P}_2, \ldots$,  all of the properties $\mathcal{P}_i$ are \emph{simultaneously} generic in $X$. This is because the family of residual sets is closed under countable intersections.

Now for some ergodic definitions and notation. By ``measure'' we always refer to non-signed measures defined on the Borel $\sigma$-algebra of the ambient manifold $M$. We denote by $\mm$ the set of all probability measures (more succintly referred to as "probabilities") on $M$, and by $\mfm$ the set of all $f$-invariant probabilities on $M$. Both of these spaces are endowed with the usual weak* topology, turning them into compact metrizable spaces. We fix and denote by $m$ a volume probability on $M$ which we simply refer to as ``Lebesgue measure'' (see remark \ref{remark on Lebesgue} regarding the relevance of volume measures).

\smallskip

Given a continuous dynamical system $f: M \to M$ and a point $x \in M$, the \emph{Birkhoff limit} of the point $x$, when it exists, is given by the probability measure $\mu_x = \lim_{n \to \infty} \; \frac{1}{n} \sum_{j=1}^{n} \delta_{f^j(x)}$, where $\delta_y$ denotes the one-point Dirac probability supported on $y$ and the limit is taken in the weak* topology. When this limit exists, it is a fortiori an $f$-invariant probability. The Birkhoff limit $\mu_x$ is characterized by the following condition: given any continuous function $\varphi: M \rightarrow \R$, the average $\lim_{n \to \infty} \; \frac{1}{n} \sum_{j=1}^{n} \varphi(f^j(x))$ coincides with the integral $\int_M \varphi \; d\mu_x$.

A map $f$ is \emph{totally singular} (with respect to Lebesgue measure) if there is a (Borel) measurable set $\Lambda$ such that $m(\Lambda) = 1$ and $m(f^{-1}(\Lambda)) = 0$.

\begin{definition}\label{physical}
Given a probability measure $\mu$ on $M$ we define its basin to be the set  
\[B(\mu) = \{x \in M : \mu_x \text{ \emph{is well defined and coincides with} } \mu \}.\] 
A probability $\mu$ is called a physical measure if $B(\mu)$ has positive Lebesgue measure.
\end{definition}

\begin{remark}
Only invariant measures can have non-empty basin. In particular, every physical measure is invariant. It is not true, however, that physical measures are necessarily ergodic (see \cite{MR2373211} for a counterexample).
\end{remark}

\smallskip

Given a periodic point $p$ of period $k$, its orbit supports a unique invariant probability, the \emph{periodic Dirac measure}, given by $\mu_p = \frac{1}{k} \sum_{j=0}^{k-1} \delta_{f^j(p)}$. Note that this measure coincides with the Birkhoff limit of $p$, so that there is no ambiguity in the notation we employ.

\smallskip

We recall that, given a probability $\mu$ on $M$, its \emph{push-forward} under $f$ is the probability $f_*\mu$ defined by $f_*\mu(A) = \mu(f^{-1}(A))$ for every Borel measurable set $A \subset M$.

Finally, some new vocabulary regarding ergodically pathological (or well-behaved) dynamics:

\begin{definition}\label{definitions_weird}
A dynamical system $f: M \to M$ is said to be

\begin{itemize}
\item[w1)]
\emph{wonderful} if there exists a finite or countable family of physical measures $\mu_n$ such that $$m \left(\bigcup_n \; B(\mu_n) \right) = 1;$$

\item[w2)]
\emph{wholesome} if $m$-a.e. point $x$ has a well-defined Birkhoff limit $\mu_x$;

\item[w3)]
\emph{weird} if $m$-a.e. point $x$ has a well-defined Birkhoff limit $\mu_x$, but $f$ is totally singular and moreover admits no physical measure;

\item[w4)]
\emph{wacky} if $m$-a-e. point $x$ does \emph{not} have a well-defined Birkhoff limit; and

\item[w5)]
\emph{wicked} if $f$ is not uniquely ergodic and moreover the averages 
\begin{equation*}
m_n = \frac{1}{n} \sum_{k=1}^{n} f_*^k m
\end{equation*} 
accumulate on the whole space of $f$-invariant measures: that is, if $$\mfm = \bigcap_{n \geq 0}  \overline{\bigcup_{k \geq n} m_k}.$$
\end{itemize}
\end{definition}

The five conditions above are set out in a roughly well-behaved-to-pathological order. Some implications between them are obvious: for example, every wonderful system is wholesome; a wacky system cannot be wholesome, nor can it admit any physical measures; and every weird system is wholesome but not wonderful. It is not obvious, yet true, that every wicked system is wacky. In Section \ref{sectionwonderful} we discuss further these conditions and the various implications between them.

\begin{remark}
We exclude uniquely ergodic systems from the wicked ones because otherwise every uniquely ergodic system would be both wicked and wonderful --- which would sound great in a Lewis Carroll novel but impair our attempt to set the definitions in a roughly well-behaved-to-pathological order.
\end{remark}

\begin{theorem} \label{theoA}
Let $\cC$ denote either (i) the space $C^0(M)$ of all continuous maps from $M$ to itself, where $\dim(M)$ is arbitrary, or (ii) the space $\Hom(M)$ of all homeomorphisms from $M$ to itself, where $\dim M \geq 2$ (i.e., $M$ is not the circle). Then there is a residual set $\cR \subset \cC$ such that every $f \in \cR$ is weird.
\end{theorem}

The behavior above shows that, from the point of view of view of Lebesgue measure, generic continuous maps (except for circle homeomorphisms) have very complicated global dynamics, but most individual orbits are quite well-behaved; this nicely parallels the conclusion of Hurley et al's book \cite{MR1980335} regarding the behavior of orbits of Baire-residual points. In their words: ``the dynamics of a generic homeomorphism is geometrically complicated but the behavior of most orbits is quite stable'' (page $5$). 

\begin{remark}\label{remark on Lebesgue}
We feel that a comment on the smoothness of $M$ and the role of the Lebesgue measure is due at this point. Starting with the latter, the careful reader might have asked himself wether it is natural to use a volume measure as a reference when studying homeomorphisms since these do not respect the class of volume measures. On the contrary, generic maps and homeomorphisms are totally singular with respect to any volume measure. Indeed, there \emph{is} nothing special about volume measures. By conjugating the residual set $\mathcal{R}$ in Theorem $\ref{theoA}$ with any homeomorphism, it becomes clear that the statement remains true if the volume measure is replaced by any member of its homeomorphic class $[m] = \{ h_* m: h \in \Hom(M) \}$. By the Homeomorphic Measures Theorem, $[m]$ consists of precisely those measures which have no atoms and full support, the measure class referred to in \cite{MR1826331} as OU measures (from Oxtoby-Ulam). We may pick any measure in this class and call it "Lebesgue". Once we have abandoned the notion of volume measure, one may be justly sceptic about the hypothesis that $M$ be a smooth manifold. In fact, we chose to work with smooth manifolds because we know that these have triangulations (which are used in the proof of Theorem \ref{theoA}) and because we (the authors) do not have sufficient familiarity with non-smoothable manifolds to determine the most general structure necessary on $M$ for our proof to work.  One of our anonimous referees has sugested that the natural class to work with would be that of piecewise linear manifolds. We believe this to be correct, and express our gratitude for suggesting it. 

Another interesting thing to notice is that OU is generic in $\M(M)$ (see \cite{MR0457675}). It is therefore an exercise to see that Theorem \ref{theoA} can alternatively be formulated like this: For a generic pair $(f, \nu)$ in $\Hom(M) \times \M(M)$, $f$ is totally singular with respect to $\nu$ and the Birkhoff averages along orbits of $f$ under any continuous function converge $\nu$-almost everywhere. Moreover, $\nu(B(\mu))=0$ for every $\mu \in \M_f(M)$.
\end{remark}

The ergodic behavior of generic homeomorphisms of the circle is very different from the scenario of Theorem \ref{theoA}. Indeed, from an ergodic point of view, these are utterly well-behaved:

\begin{theorem}\label{theoB}
There is a residual set $\cR \subset \Hom(S^1)$ of all circle homeomorphisms such that every $f \in \cR$ is wonderful. Moreover, the set of physical measures is countably infinite and each physical measure is a periodic Dirac measure.
\end{theorem}

\begin{remark}

The collection of physical measures whose existence is assured by Theorem \ref{theoB} enjoy a peculiar form of robustness: Given any of the physical measures $\mu$  and any $\eps>0$ there is an open neighborhood $\cU$ of $f$ in $\Hom(S^1)$ such that if $g \in \cU$ then $g$ has a Dirac physical measure $\nu$ whose support is $\epsilon$-close to the support of $\mu$ in the Hausdorff topology on $M$. Moreover, the basins of $\mu$ and $\nu$ are close in the the sense that $m(B(\mu) \bigtriangleup B(\mu^g)) < \eps$. One is thus tempted to conclude that, when looked at individually, each physical measure of $f\in \cR$ admits a weak* continuation. However, continuation is not the right word in this context, for two reasons. Firstly, for a generic element of $\Hom(S^1)$, no physical measure is isolated, i.e. it is accumulated on by other physical measures, both in terms of its support and, consequenly (since they are all periodic Dirac measures of the same period), in the weak* topology. The other reason is that even if $f \in \Hom(S^1)$ has an isolated Dirac physical measure, e.g. as in the (non-generic) case of Morse-Smale diffeomorphisms, one can easily perturb $f$ in the $C^0$ topology to obtain a new homeomorphism with an arbitrarily large number of physical measures near it. 

The type of continuity that does hold, however, is that the weak* closure of the set of physical measures depends continuously on $f$ in the Hausdorff topology in $\mfm$ whenever $f \in \cR$. 

\end{remark}

Although the ergodic properties of the generic systems studied in Theorems \ref{theoA} and \ref{theoB} are radically different from the global viewpoint, from the point of view of individual orbits they are quite similar in that they are wholesome --- Lebesgue almost every orbit has a well defined Birkhoff limit. This is essentially due to an abundance, in both of these contexts, of \emph{trapping regions}: open regions which are mapped strictly into their own interiors. We believe that the abundance of trapping regions is essentially equivalent to weirdness, and that in their absence wickedness will prevail. By a famous theorem of Conley \cite{MR511133}, sometimes referred to as the fundamental theorem of dynamical systems (see \cite{MR1366526} for the discrete non-invertible case), a continuous map has no trapping region if and only if it is chain recurrent.

\begin{conjecture}
 Let $\mathcal{C}$ be either (i) the set of chain recurrent continuous maps on $M$, with $M$ of any dimension, or (ii) the set of chain recurrent homeomorphisms on $M$, with $M$ of dimension at least $2$. Then a generic element of $\mathcal{C}$ is wicked. (Note that, in either case, $\mathcal{C}$ is a Baire space since it is a closed subset of the complete metric space $C^0(M)$ or $\Hom(M)$ respectively.)
\end{conjecture}

Our third and final result points in this direction, proving the conjecture to be true in the context of circle maps which are conjugated to expanding ones. More precisely, given an integer $k$ with $\vert k \vert \geq 2$, let $E_k$ denote the linearly induced expanding circle map of degree $k$, i.e. the map $x \mapsto k x \mod 1$. 

Consider the sets
\[
CE_k(S^1) = \{h E_k h^{-1}: h \in \Hom_+(S^1)\}.
\]
Thus $CE_k(S^1)$ consists of all continuous circle maps which are topologically conjugated to $E_k$; hence to any $C^1$ map $f$ of degree $k$, having an iterate $f^n$ such that $\vert (f^n)'(x) \vert >1$ for all $x \in S^1$ --- the standard definition of expanding circle map. Likewise, the set 
\[
CE(S^1) = \bigcup_{\vert k \vert \geq 2} CE_k(S^1)
\]
consists of all continuous circle maps topologically conjugate to some expanding map. It becomes a topological space by considering it as a subspace of $C^0(S^1)$. As such, it is neither closed nor open. It is a nowhere dense set which, by any reasonable standard, should be considered extremely meager. Still, by Proposition \ref{locally homeomorphic}, $CE(S^1)$ is itself a Baire space, so it becomes relevant to ask what its generic properties are.

\begin{theorem}\label{theoC}
Generic elements of $CE(S^1)$ are wicked.
\end{theorem}

So in this context generic dynamics is even more pathological: iteration of Lebesgue measure completely ``deforms'' it. In fact, a simpler argument than the one employed in the proof of Theorem \ref{theoC} proves that, for a generic $f$ in $CE(S^1)$, the induced push-forward map $f_*$ is transitive on $\M(M)$, having $\{f_*^n m: m \geq 0 \}$ as a dense orbit. The proof of Theorem \ref{theoC} is slightly more involved because it has to deal with the extra difficulty of showing that, when $f_*^n m$ gets near an invariant measure $\nu$, it stays there long enough so that $\frac{1}{n}\sum_{k=0}^{n-1} f_*^k m$ gets near $\nu$. 

We end this section with a few remarks:

\begin{itemize}
\item
One of the central themes of ergodic theory is of course entropy. We remark that by \cite{MR579700} every $C^0$-generic homeomorphism (except of course for circle homeomorphisms) and every $C^0$-generic continuous map has infinite topological entropy.

\item Apart from the trivial case of Morse-Smale diffeomorphisms, very little about generic existence of physical measures is known in the $C^1$ topology. Campbell and Quas \cite{MR1845327} used an approach based on thermodynamic formalism to prove that a generic $C^1$ expanding circle map has a unique physical measure. Their argument was recently adapted to $C^1$ generic hyperbolic attractors \cite{MR2770016}. It was also shown in  \cite{MR2811152} that ``tame'' $C^1$-generic diffeomorphisms -- which include transitive ones -- exhibit a Baire-residual subset $S$ of $M$ such that the Birkhoff average $\mu_x$ exists for every $x \in S$: that is, $C^1$-generic tame diffeomorphisms are ``Baire-wholesome''. But as far as the authors know there are no results on the wholesomeness or wickedness of Lebesgue-a.e. point of the manifold in this context. 

For higher regularity ($r > 1$), it is a classical result that Anosov diffeomorphisms, or more generally, Axiom A diffeomorphisms with no cycles, are open sets of wonderful maps. The same holds for uniformly expanding maps. Being open, these sets intersect any residual set. Efforts have been made to enlarge these sets to certain classes of diffeomorphisms admitting dominated splitting, by assuming non-uniform contraction or expansion in one of the invariant directions (instead of uniform contraction and expansion, which is the case for Axiom A systems). See \cite{MR1749677, ABV, MR2574879}. There is a result due to Tsujii \cite{MR2231338} which states that a $C^r$ generic partially hyperbolic endomorphism on the 2-torus is wonderful, whenever $r \geq 19$. It is the most remarkable result in this direction in terms of technical sophistication. 

\item
Another interesting question, raised independently by Ch. Bonatti and by E. R. Pujals in private discussions, is whether $C^0$-\emph{densely} the dynamics is finitely wonderful (i.e., there is a finite set of physical measures the union of whose basins has full Lebesgue measure); this is a $C^0$-version of the Palis conjecture on finitude of attractors \cite{MR1755446}. A partial (positive) answer to this question in the context of homeomorphisms is given by combining results of Moise, Shub, and Sinai-Ruelle-Bowen: the ``$C^0$-Palis conjecture for homeomorphisms'' holds in dimensions $d = 1, 2, 3$, where homeomorphisms can always smoothed by $C^0$-perturbations into diffeomorphisms (see \cite{MR0488059}, which in turn can (by \cite{MR0307278}) be $C^0$-perturbed into structurally stable $C^{\infty}$ diffeomorphisms, which are finitely wonderful by \cite{MR2423393}). We note that in dimensions $d \geq 7$ there do exist non-smoothable homeomorphisms, by \cite{MR0082103}.

\end{itemize}

\section{Wonderful, wholesome, weird, wacky, wicked}\label{sectionwonderful}

In this section we first discuss the ``w'' conditions defined in Definition \ref{definitions_weird} and the implications among them, and then point out some examples.

\subsection{Implications}

Some of the implications among the five w's are immediate or trivial and indeed have already been mentioned in the Introduction. They are:

\begin{itemize}
\item
Every wonderful system is wholesome.

\item
Wacky system are not wholesome; moreover they admit no physical measures.

\item
Every weird system is wholesome but not wonderful.

\item
No weird system is wacky.

\end{itemize}

A less obvious implication is

\begin{prop}
Every wicked system is wacky.
\end{prop}

\begin{proof}
Suppose that $f$ is not wacky. That is, that there exists a set $A$ of positive Lebesgue measure such that for every $x\in A$, the Birkhoff limit $\mu_x = \lim_{n \rightarrow \infty} \frac{1}{n} \sum_{k=0}^{n-1} \delta_{f^k(x)}$ is well-defined. We shall prove that, in this case, $f$ cannot be wicked. If $f$ is uniquely ergodic, there is nothing to prove. Suppose it is not; then there exist two distinct \emph{ergodic} $f$-invariant measures $\nu_1, \nu_2$. To prove the proposition, it suffices to prove that $m_n = \frac{1}{n} \sum_{k=0}^{n-1}f_*^k m$ cannot accumulate on both $\nu_1$ and $\nu_2$.

Let $m_A$ denote the normalized restriction of Lebesgue measure to $A$. Then the limit
\begin{equation}
 \mu_A = \lim_{n \rightarrow \infty} \frac{1}{n} \sum_{k=0}^{n-1} f_*^k m_A
\end{equation}
is well-defined. In fact, the limit is the unique measure $\mu_A$ such that
$\int \varphi d\mu_A = \int ( \int \varphi d\mu_x) dm_A$ holds for every continuous $\varphi : M \rightarrow \mathbb{R}$. To see this, observe that
\begin{align}
 \lim_{n \rightarrow \infty} \int \varphi \  d\left( \frac{1}{n} \sum_{k=0}^{n-1} f_*^k m_A\right)
& = \lim_{n \rightarrow \infty} \int  \frac{1}{n} \sum_{k=0}^{n-1} \varphi(f^k(x)) \ dm_A(x) \label{line1} \\
& = \int \left( \int \varphi \ d\mu_x \right) \ dm_A(x) = \int \varphi \ d\mu_A, \label{line2}
\end{align}
where the passage from (\ref{line1}) to (\ref{line2}) is justified by the dominated convergence theorem. 

Note that if $A$ has full Lebesgue measure, then we are done because then $\mu_A$ will be the only accumulation point of $m_n$. For the remainder of the proof we therefore suppose that $A^c$ has positive Lebesgue measure. We denote by $m_{A^c}$ the normalized restriction of Lebesgue measure to $A^c$.
To prove that $m_n$ cannot accumulate on both $\nu_1$ and $\nu_2$, we start by fixing some $0< \epsilon < m(A)/2$. Since $\nu_1$ and $\nu_2$ are mutually singular, there exists a continuous function $\varphi:M \rightarrow [0,1]$ such that
$\int \varphi \ d\nu_1 <\epsilon$ and $\int \varphi \ d\nu_2 > 1-\epsilon$. Thus if $m_n$ were to accumulate on both $\nu_1$ and $\nu_2$ we would have
\begin{equation}\label{gap}
 \limsup_{ n \rightarrow \infty} \int \varphi dm_n - \liminf_{n \rightarrow \infty} \int \varphi dm_n > 1-2\epsilon > \mu(A^c).
\end{equation}
However, observing that
\begin{equation}
 \frac{1}{n} \sum_{k=0}^{n-1} f_*^k m = m(A^c) \frac{1}{n} \sum_{k=0}^{n-1} f_*^k m_{A^c} + m(A) \frac{1}{n} \sum_{k=0}^{n-1} f_*^k m_A,
\end{equation}
we estimate
\begin{equation}
 \limsup_{n \rightarrow \infty} \int \varphi dm_n \leq m(A^c) + m(A) \int \varphi d\mu_A
\end{equation}
and
\begin{equation}
  \liminf_{n \rightarrow \infty} \int \varphi dm_n \geq  m(A) \int \varphi d\mu_A.
\end{equation}
Therefore $\limsup_{ n \rightarrow \infty} \int \varphi dm_n - \liminf_{n \rightarrow \infty} \int \varphi dm_n \leq m(A^c)$, contradicting (\ref{gap}).

\end{proof}

In order to make the grand scheme of things clearer, we include an Euler diagram:

\vspace{1cm}

\begin{tikzpicture}
\draw (-3,-2.5) rectangle (8,2.8);
\draw (6.2,2.2) node {$C^0(M)$};

\draw (0,0) ellipse ( 2.5 and 1.7 );
\draw (0,1.2) node {$w2$};

\draw (-1,0) ellipse ( 0.9 and 0.9 );
\draw (-1,0) node {$w1$};

\draw (1.1,0) ellipse ( 0.7 and 0.7 );
\draw (1.1,0) node {$w3$};

\draw (5,0) ellipse (1.5 and 1.5);
\draw (4.6,1) node {$w4$};

\draw (5.3,-0.2) ellipse (0.5 and 0.5);
\draw (5.3,-0.2) node {$w5$};
\end{tikzpicture}

\subsection{Examples}

There are of course many well-known examples of systems which are wonderful (e.g, uniformly hyperbolic ones). We do not know of any example of weird systems in the literature. There are many examples of maps which have convergent Birkhoff averages Lebesgue almost everywhere and yet have no physical measures. The identity on any manifold or rational translations on tori are examples of such. We do not think these are weird at all and that is why we included the requirement of being totally singular into the definition of weird. Though as far as the literature contains no examples of weird dynamics, Theorem \ref{theoA} shows that weirdness is extremely abundant in the $C^0$ topology.  

Rigid periodic systems such as rational circle rotations are examples of wholesome systems thet are neither wonderful nor weird. Other types may easily be constructed by combining, for example, the identity and the map $f(x) = \frac{1}{10}\sin^2(2\pi x)$ on $S^1$ --- simply use the identity on the first half of the circle and $f$ on the other.

Some authors \cite{MR2180226, MR2122689} study the notion of natural measures. It is usually defined as a measure $\mu$ such that, given any measure $\nu$ absolutely continuous with respect to Lebesgue measure, we have 
\begin{equation}\label{natural}
\frac{1}{n}\sum_{k=0}^{n-1} f_*^k \nu \rightarrow \mu. 
\end{equation}
 Sometimes it is only required that (\ref{natural}) hold for all $\nu$ supported in the basin of a given topological attractor.
If a continuous map has a unique physical measure whose basin is of full Lebesgue measure, then this measure is also natural. However, it was proved in \cite{MR1950793} that there are natural measures that are not physical. For some time it remained unclear whether it could be true that every ergodic natural measure is physical, until this was proved in \cite{MR2122689} not to be the case. Misiurewicz \cite{MR2180226} gives an example of a continuous map $f: \mathbb{T}^2 \rightarrow \mathbb{T}^2$ having a natural measure, but with the pathological property that, for Lebesgue almost every $x\in \mathbb{T}^2$, the sequence $\frac{1}{n}\sum_{k=0}^{n-1} \delta_{f^k(x)}$ accumulates on the whole of $\mfm$, which, in Misiurewicz's example, is a very rich set. In particular, his example proves that there are systems that are wicked but not wacky.

\section{Proof of Theorem \ref{theoA}}\label{sectiontheoA}

Theorem \ref{theoA} is a fairly straightforward consequence of the following Lemma:

\begin{shred}
Let $\cC$ be be either the space $C^0(M)$ of all continuous maps from $M$ to itself with $dim(M)\geq 1$ or the space $\Hom(M)$ of all self-homeomorphisms of $M$ with $\dim(M) \geq 2$. Given any $\eps > 0$, there is a dense subset $\cC^{\eps} \subset \cC$ such that for every $f \in \cC^{\epsilon}$ there is a family of pairwise disjoint open sets $U_1, \ldots, U_{N}$ such that   

\begin{enumerate}[i)]

\item each $U_j$ is a trapping region: $f(\overline{U_j}) \subset U_j$ for every $j \in \{1, \ldots N \}$;

\item each $U_j $ has small Lebesgue measure: $$m(U_j) < \eps;$$

\item
the union of the sets $U_j$ occupies, Lebesgue-wise, most of $M$: $$m(\bigcup_{j=1}^N  \; U_j) > 1 - \eps;$$

\item  the sets $U_j$ are ``crushed'' by iteration: 
\begin{equation}\label{crushing}
m(f(U_j)) < \epsilon \cdot m(U_j)
\end{equation}

\item each $U_j$ is strictly contained in the basin of a periodic cycle of sets with small diameter: there exist open sets $W_j^1, \ldots, W_j^{k_j}$ such that 
\begin{enumerate}
\item $\diam(W_j^i)< \epsilon$ for every $i \in \{1, \ldots, k_j\}$,
\item $\overline{f(W_j^i)} \subset W_j^{i+1}$ for every $i\in \{1, \ldots, k_j-1\}$, \\ $\overline{f(W_j^{k_i})} \subset W_j^1$, and 
\item 
\begin{equation}
\overline{U_j} \subset \bigcup_{n \geq 0} f^{-n}(W_j^1 \cup \ldots \cup W_j^{k_j});
\end{equation}
\end{enumerate}

\end{enumerate}

\end{shred}

We first show how the Shredding Lemma implies Theorem \ref{theoA}, and later prove the Shredding Lemma itself.

\begin{proof}[Proof of Theorem \ref{theoA}]
First note that all five conclusions of the shredding lemma are robust under small $C^0$ perturbations so that the sets 
$\cC^\epsilon$ are, in fact, open and dense. 
Let $\cR$ be the residual set obtained by intersecting a countable number of these:

 $$\cR = \bigcap_{n \in \NN} \; \cC^{\frac{1}{n}}.$$

We claim that every $f \in \cR$ is weird. The easiest part is to prove that elements of $\cR$ can have no physical measures. This is a consequence of the observation that if $f \in \cC^{\epsilon}$ then 
\begin{equation}
\sup_{\mu \in \cM_f} m(B(\mu)) < 2 \epsilon.
\end{equation}
Indeed, it follows from properties ii) and iii) of the Shredding Lemma that if there were a measure $\mu$ with $m(B(\mu)) \geq 2 \epsilon$, then there would be points $x$ and $y$ of $B(\mu)$ which belong to different sets $U_j$ and $U_{j'}$. But a bump function that takes the value $1$ at $U_j$ and $0$ at $U_{j'}$ clearly shows that the points $x,y$ cannot lie in the basin of the same measure, a contradiction.

To see that every $f \in \cR$ is totally singular we make use of a little trick. Recall that $\cR$ was obtained as a a countable intersection of sets $\cC^{\frac{1}{n}}$. For each $n \in \NN$ let $V_n$ be the union of the trapping regions from the Shredding Lemma. Then $m(V_n)> 1-1/n$ and $m(f(V_n)) < 1/n$. Let 
\begin{equation}
\Lambda = \bigcap_{k \geq 1} \bigcup_{n \geq k} V_{n^2}.
\end{equation}
(The trick is to consider $V_{n^2}$ rather than $V_n$.) Observe that $\Lambda$ has full Lebesgue measure and that $m(f(\Lambda)) \leq \sum_{n=k}^{\infty} \frac{1}{n^2}$ for every $k \geq 1$, by the crushing property (\ref{crushing}). Hence $f(\Lambda) = 0$. The set $f(\Lambda)^c$, therefore, has total Lebesgue measure but its pre-image under $f$ has zero Lebesgue measure. Hence $f$ is totally singular.

It remains to show that elements of $\cR$ have convergent Birkhoff averages Lebesgue almost everywhere. Thus we fix $f\in \cR$ and consider the set 
\begin{equation}
\Lambda= \bigcap_{k \geq 1}\bigcup_{n\geq k} V_n.
\end{equation} 
(We could still work with $\Lambda= \bigcap_{k \geq 1}\bigcup_{n\geq k} V_{n^2}$ but for the current purpose it is not a very natural choice as we have no longer any need for summabillity.) Fix a continuous function $\varphi: M \to \RR$ and consider some point $x \in \Lambda$. We will show that, for every $\delta > 0$,
\begin{equation}\label{small variation}
\limsup_{n} \; \frac{1}{n} \sum_{m=1}^{n} \varphi(f^m(x)) - \liminf_{n} \; \frac{1}{n} \sum_{m=1}^{n} \varphi(f^m(x)) \leq  2 \delta.
\end{equation}
Once that is done, the proof is complete.

Since $\varphi$ is uniformly continuous, for large enough $n_0 \in \NN$ it holds that $|\varphi(x) - \varphi(y)| < \delta$ whenever $d(x, y) < \frac{1}{n_0}$. Now, by the definition of $\Lambda$,  there is some iterate $f^n(x)$ of $x$ belonging to the some set $W$ with  diameter smaller than $ \frac{1}{n_0}$ and such that $f^{k}(\overline{W}) \subset W$ for some $k \geq 1$. Since the veracity of (\ref{small variation}) does not change if $x$ is replaced with some iterate of itself, we might assume that $x\in W$ for simplicity. Observe that $d(f^{\ell \cdot k+r}(x), f^r(x))<1/n_0$ and, consequently, $|\varphi(f^{\ell \cdot k +r }(x))- \varphi(f^r(x))| < \delta$ for every $\ell \in \mathbb{Z}$. Thus, writing an arbitrary integer $n \geq 0$ as $n = \ell \cdot k + r$ with $0 \leq r < k$ and $\Gamma = \frac{1}{k} \sum_{m=0}^{k-1} \varphi(f^m(x))$ we obtain the estimate
\begin{equation}
\Gamma - \delta  - \frac{r}{n} \cdot \|\varphi\|_{C^0} < \frac{1}{n}\sum_{m=1}^{n} \varphi(f^m(x)) < \Gamma + \delta + \frac{r}{n} \cdot \|\varphi\|_{C^0},
\end{equation} 
of which (\ref{small variation}) is a direct consequence.
\end{proof}

We prove the Shredding Lemma for homeomorphisms and continuous mappings separately. Our proof relies on the existence of triangulations, and is the reason why we consider smooth manifolds in the first place\footnote{See remark \ref{remark on Lebesgue}}. There may be nothing fundamental about this. In fact, we find it likely that Theorem \ref{theoA} is generalizable to non-smoothable topological manifolds that do not admit triangulations (in which case the role of Lebesgue measure could be represented by any non-atomic measure positive on open sets). However, the decomposition of the manifold into simplices provides a very handy set of coordinates defined on convex subsets of $\mathbb{R}^n$ in which we can perform the explicit perturbations used in the proof. 

Thus by a "\emph{triangulation}" we simply mean a finite collection $\mathcal{R} = \{R_i \}_{i \in I}$ of compact subsets of $M$ homeomorphic to (say) the simplex
\begin{equation}
\Delta_d = \{(x_1, \ldots, x_d) \in \mathbb{R}^d:  \sum_{i=1}^d x_i \leq 1 \text{ and }x_i \geq 0 \text{ for all }i \},
\end{equation}
together with homeomorphisms $\xi_i: R_i \rightarrow \Delta_d$, such that $m(\partial R_i)=0$ for every $i$ and such that if $i \neq j$ then (i) either $R_i \cap R_j = \emptyset$ or (ii) $R_i \cap R_j \subset \partial R_i$. Cairns triangulation \cite{MR1503083} does the job. The diameter of a subset $A \subset M$ is defined as $\diam(A) = \sup_{x,y \in A} d(x,y)$ and the diameter of a triangulation $\mathcal{R} = \{R_1 \}_{i \in I}$ is $\diam \mathcal{R} = \max \{ \diam(R_i): i  \in I \}$. We shall say loosely that a triangulation is \emph{fine} if its diameter is small. By dividing the standard simplex into smaller simplices we may subdivide any triangulation into a finer one. In particular, there are arbitrarily fine triangulations of $M$.

\begin{proof}[Proof of the Shredding Lemma for homeomorphisms in higher dimensions]

We fix an arbitrary $f \in \Hom(M)$, and $\epsilon>0$. Consider a triangulation $\mathcal{R}=\{R_i\}_{i \in I}$ of $M$ with diameter $\epsilon_0<\epsilon$. We will construct homeomorphisms $h,k \in \Hom(M)$ satisfying $h(R_i) = R_i$ and $k(R_i) = R_i$ for every $i \in I$ (hence $d(h, \id) < \epsilon_0$ and $d(k,\id)<\epsilon_0$) such that $g = k f h$ satisfies properties i) to v) in the Shredding Lemma. Then, by uniform continuity of $f$,  $d(g, f) < \epsilon$ provided that $\epsilon_0$ is sufficiently small. Throughout the proof we use the notation $A \subset \subset B$ to express that $\overline{A} \subset B$ and $|S|$ to denote the cardinality of a finite set $S$.

Let $\tau:I \rightarrow I$ be any map such that $\interior(f(R_i)) \cap \interior ( R_{\tau(i)}) \neq \emptyset$ for every $i \in I$. Once such a map $\tau: I \rightarrow I $ is chosen we pick, for each $i \in I$, some point  $p_i \in \interior (R_i)$ with the property that $ f(p_i) \in \interior(R_{\tau(i)})$.
For $\delta>0$ we write $R_i^{\delta}$ for the set $\{x \in R_i: d(x,\partial R_i)>\delta\}$. We can (and do) choose $\delta>0$ small enough that $p_i \in R_i^{\delta} $ and $f(p_i) \in R_{\tau(i)}^{\delta}$ for every $i \in I$. 

Instead of proceeding with the full proof of the Shredding Lemma, we take a detour and show how we would go about if we were to prove only items i), iii), iv) and v). i.e. the part of the Lemma responsible for total singularity and almost everywhere convergence of Birkhoff averages of generic homeomorphisms. By doing so, we save time for the reader who is mostly interested in these topics as well as provide a warm-up for the rest.

The procedure is rather simple. Take $\delta'>0$ small enough that we have $f(B_{\delta'}(p_i)) \subset \subset R_{\tau(i)}^{\delta}$ for every $i \in I$. We wish to construct a homeomorphism $h:M \rightarrow M$ that sends each $R_{i}^{\delta}$ into $B_{\delta'}(p_i)$. This is easily done using the linear structure in each $R_ì$ induced by the charts $\xi_i: R_i \rightarrow \Delta^d$. 

An explicit choice can be obtained by taking 
\begin{equation}\label{perturbations}
h_i (x) = 
\begin{cases}
x & \text{ if } x \notin  R_i \\
\alpha(x)^T x +(1-\alpha(x)^T) p_i & \text{ if } x \in R_i,
\end{cases}
\end{equation}
for a sufficiently large $T>0$, where
\begin{equation}
\alpha(x) = \frac{d(x, p_i)}{d(x, p_i) + d(x, \partial R_i)},
\end{equation}
and let $h$ be the composition of the $h_i$. (Since the $h_i$ commute, the order of composition is irrelevant.)
One verifies that, no matter how small we take $\delta'>0$, we may always choose $T>0$ sufficiently large to obtain 
\begin{equation}
h(R_i^{\delta}) \subset \subset B_{\delta'}(p_i)
\end{equation}
and consequently 
\begin{equation}
fh(R_i^{\delta}) \subset \subset R_{\tau(i)}^{\delta}.
\end{equation}

The choice of the sets $U_j$ that needs to be made in order to them to satisfy items i), iii), iv) and v) of the Shredding Lemma depends on the dynamics of the map $\tau: I \rightarrow I$. 
Indeed, if $\tau:I \rightarrow I$ happens to be a cyclic permutation, then the  choice 
\begin{align} 
U_1 & = \bigcup_{i \in I} R_i^{\delta} \text{ and} \\
W_1^i & = R_{\tau^i(\alpha)}^{\delta}, 
\end{align}
where $\alpha$ is any element of $I$, will do.
On the other extreme, if $\tau$ happens to be the identity on $I$, we would simply take 
\begin{equation}
U_j = W_j^1 = R_j^{\delta}
\end{equation}
for every $ i \in I$.
For the general case we observe that, since $I$ is a finite set, it has at most a finite number of periodic orbits $O_1, \ldots O_N$, and that every $i \in I$ ends up in one of the $O_j$ after at most $|I|$ iterates. Thus the sets $\tilde{O}_j = \bigcup_{k=1}^{|I|} \tau^{-k}(O_j)$, $j=1, \ldots, N$ form a partition of $I$. The trapping regions that satisfy i), iii), iv) and v) of the Shredding Lemma are
\begin{equation}
U_j = \bigcup_{i\in \tilde{O}_j} R_i^{\delta}
\end{equation}
with $W_j^i = R_{\tau^i(\alpha_j)}^{\delta}$ for some choice of $\alpha_j \in O_j$ so that $k_j= |O_j|$ for every $j=1, \ldots, N$.

\subsection*{Pizza Slice}

In order to obtain item ii) of the Shredding Lemma, we need to improve on our construction. Note that, up to this point, we have not made use of the hypothesis that  $M$ is at least two dimensional. Yet Theorem \ref{theoB} tells us that any attempt to prove the Shredding Lemma in the case where $\dim M=1$ must inevitably fail. In our approach, what fails is the following pizza-like decomposition of the $R_i$ into subsimplexes $\{R_{ij}\}_{j \in J}$ which has to be done in such a way that the point $p_i$ is a vertex of each of the $R_{ij}$ (see figure \ref{slicing}). A convenient way to accomplish this is to decompose each face of our standard simplex $\Delta_d$ into a large number of $(d-1)$-subsimplexes and let each of these be a face of a $d$-simplex by joining the point $\xi_i (p_i)$ (recall that $\xi_i$ are the charts $R_i \rightarrow \Delta_k$).  Then we take the $R_{ij}$ to be the pre-image of such $d$-simplex under $\xi_i$. We think of each $R_{ij}$ as a "pizza slice" of the "pizza" $R_i$ due to a certain resemblance in the case when $d=2$.
 The important feture of the pizza-slice decomposition is that
 \begin{itemize}
 \item each $R_i$ is decomposed into the \emph{same number} of subsimplexes $R_{ij}$ so that they may be conveniently labeled as $R_{ij}$ where $(i,j) \in I \times J$,
 \item $p_i \in \partial R_{ij}$ for every $(i,j) \in I\times J$,
 \item $\partial R_{ij}$ has zero Lebesgue measure,
 \item $m(R_{ij}) < \epsilon \cdot m(R_i)$ for every $(i,j) \in I\times J$.
 \end{itemize}
 
 For each $(i,j) \in I \times J$ choose a point $p_{ij} \in \interior R_{ij}$ close enough to $p_i$ that $f(p_{ij}) \in \interior R_{\tau(i)}$. If necessary, we reduce $\delta$ so that $p_{ij} \in R_{ij}^\delta$ and $f(p_{ij}) \in R_{\tau(i)}^{\delta}$ for every $(i,j) \in I\times J$. Then choose $\delta'>0$ such that $f(B_{\delta'}(p_{ij})) \subset \subset R_{\tau(i)}^{\delta}$ for every $(i,j) \in I \times J$. We are later going to impose further conditions on the smallness of $\delta'$.
  
For each $(i,j) \in I \times J$ we define a homeomorphism $h_{ij}$ which is the identity on the complement of $\interior R_{ij}$, leaving $R_{ij}$ invariant,  and such that  
 \begin{equation}\label{sending Rij into a ball}
 h_{ij}(R_{ij}^\delta) \subset \subset B_{\delta'}(p_{ij}).
 \end{equation}
This can be done in complete analogy with (\ref{perturbations}) by using the linear structure on $R_{ij}$ induced by $\xi_i$. That is, we can take
\begin{equation}\label{hij}
h_{ij} (x) = 
\begin{cases}
x & \text{ if } x \notin  R_{ij} \\
\alpha(x)^T x +(1-\alpha(x)^T) p_{ij} & \text{ if } x \in R_{ij},
\end{cases}
\end{equation}
for a sufficiently large $T>0$, where
\begin{equation}
\alpha(x) = \frac{d(x, p_{ij})}{d(x, p_{ij}) + d(x, \partial R_{ij})}.
\end{equation}
Let $h:M \rightarrow M$ be the composition of all $h_{ij}$. The order, again, is irrelevant.
This $h$ has the property that 
\begin{equation}\label{almost there}
fh(R_{ij}^{\delta}) \subset \subset R_{\tau(i)}^{\delta} \text{ for every }(i,j) \in I \times J. 
\end{equation}
  \begin{figure}[h] 
 \center
 \includegraphics[width=5.0in]{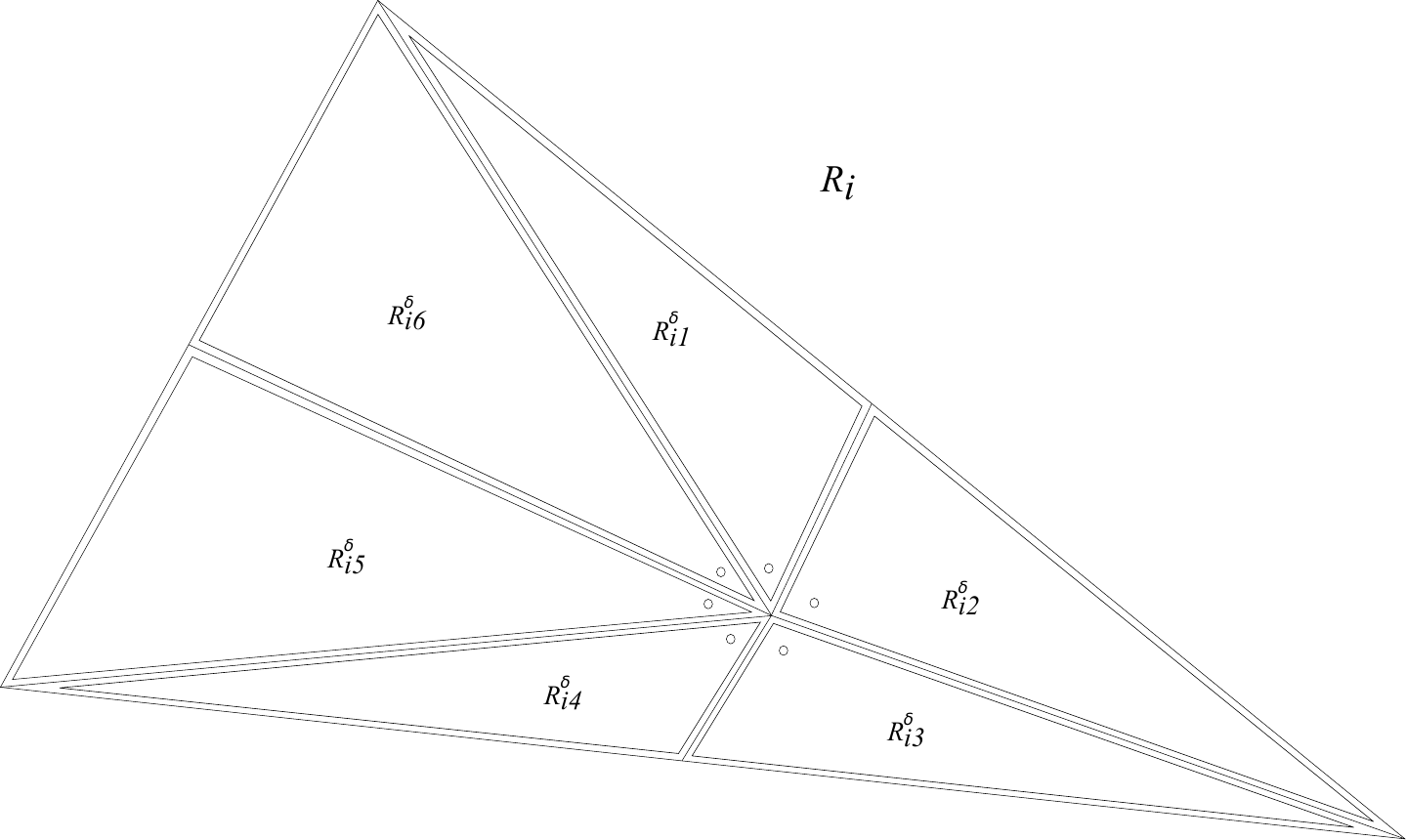}
 \caption{Pizza slice decomposition of a two dimensional simplex. Each slice is itself a simplex of dimesion two with a vertex at $p_i$. The six tiny circles in the illustration  are the balls $B_{\delta'}(p_{ij})$. Note that for small $\epsilon$ the number of slices will have to be very large (at least $\frac{1}{\epsilon}$) in order to get $m(R_{ij})<\epsilon \cdot m(R_i)$ for every $j \in J$. }
 \label{slicing}
  \end{figure}
 
 \subsection*{Tunelling}
 
 The final perturbation required to prove the Shredding Lemma is the  construction of  a homeomorphism $k:M \rightarrow M$ in such a way that 
 \begin{equation}
 kfh(R_{ij}^\delta) \subset \subset R_{\tau(i)j}^\delta \text{ for every }(i,j) \in I \times J. 
 \end{equation}
This is a significant improvement of (\ref{almost there}). It can be accomplished by composing a large number of homeomorphisms $k_{ij}$, each supported on $R_{\tau(i)j}$ (\emph{not} on $R_{ij}$). Each $k_{ij}$ is constructed as follows.

For $(i,j) \in I\times J$ pick $q_{ij} \in R_{\tau(i)j}^\delta$ in such a way that the straight line segment connecting $f(p_{ij})$ and $q_{ij}$ does not contain any of the points $f(p_{i'j'})$ nor $q_{i'j'}$ for $(i',j') \neq (i,j)$. This is possible because the $R_{ij}^\delta$ are open sets. We write $C_{ij}$ for the closed convex hull of $B_{\delta''}(f(p_{ij})) \cup B_{\delta''}(p_{ij})$ for a small number $\delta''>0$. By making sure that $\delta''>0$ is small enough we can guarantee that $B_{\delta''}(q_{ij}) \subset \subset R_{\tau(i)j}^{\delta}$ for every $(i,j) \in I \times J$ and also that 
\begin{equation} \label{no intersection}
B_{\delta''}(f(p_{i'j'})) \cap C_{ij} = B_{\delta''}(q_{i'j'}) \cap C_{ij} = \emptyset
\end{equation}
for every $(i,j) \in I\times J$ and $(i',j') \neq (i,j)$. If necessary, we reduce $\delta'$ so that $f(B_{\delta'}(p_{ij})) \subset \subset B_{\delta''}(f(p_{ij}))$. In this case we may also have to increase the value of the number $T$ in (\ref{hij}) further so that (\ref{sending Rij into a ball}) hols for each $(i,j) \in I \times J$.
 
 We define $k_{ij}$ to be a homeomorphism on $M$, leaving $C_{ij}$ invariant, with the property that $k_{ij}( f(B_{\delta'}(p_{ij}))) \subset \subset B_{\delta''}(q_{ij})$ and $k_{ij}(x)=x$ for any $x \notin C_{ij}$.
 An explicit formula for such $k_{ij}$ can be written down using the linear structure on $R_{\tau(i)}$ induced by $\xi_{\tau(i)}$:
 \begin{equation}
k_{ij} (x) = 
\begin{cases}
x & \text{ if } x \notin \interior C_{ij} \\
\beta_{ij}(x)^T x +(1-\beta_{ij}(x)^T) q_{ij} &\text{ if } x \in C_{ij},
\end{cases}
 \end{equation}
 where
 \begin{equation}
 \beta_{ij}(x) = \frac{d(x, q_{ij})}{d(x, q_{ij}) + d(x, \partial C_{ij})}
 \end{equation}
 and $T>0$ is sufficiently large.
 
 The desired homeomorphism $k: M \rightarrow M$ is now obtained by composing all the $k_{ij}$. If some of the $C_{ij}$ intersect each others, the resulting $k$ will depend on the order of composition of the $k_{ij}$. We can choose any order, as the condition (\ref{no intersection}) makes sure that 
 \begin{equation}
 kfh(R_{ij}^{\delta}) \subset \subset B_{\delta''}(q_{ij}) \subset \subset R_{\tau(i)j}^{\delta}
 \end{equation}
 for every $(i,j) \in I\times J$.

 \begin{figure}[h]
 \center
 \includegraphics[width=5.0in]{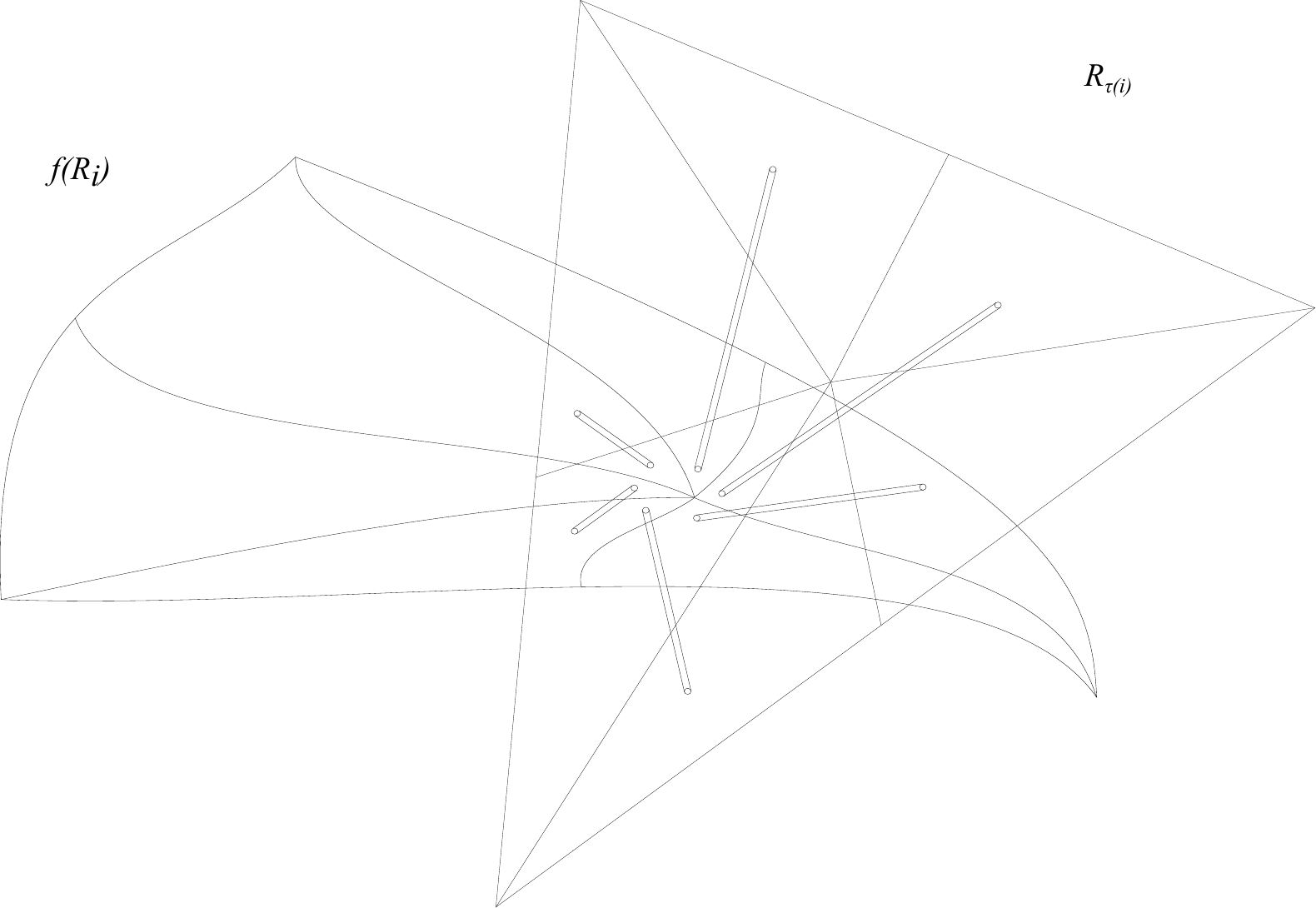}
 \caption{Illustrarion of the tunneling procedure. The six fine tubes represent the 
 $C_{ij}$. The homeomorphism $k$ is the identity outside these tubes and acts by moving all points from the sets $f(B_{\delta'}(p_{ij}))$ to their final destination in $R_{\tau(i)j}^{\delta}$.}
  \end{figure}
 
It is now easy to check that the sets
\begin{equation} \label{second try}
U_j = \bigcup_{i \in I} R_{ij}^{\delta}
\end{equation}
satisfy, not only i), iii) and iv), but also ii) of the Shredding Lemma; the neccesary ingredient for proving non-existence of physical measures. Indeed,
\begin{equation}\label{ nonexistence of pm}
m(U_j) = \sum_{i \in I} m(R_{ij}^\delta) < \sum_{i \in I} \epsilon \cdot m(R_i) = \epsilon.
\end{equation}

Alas, the sets $U_j$ in (\ref{second try}) need not satisfy item v) of the Shredding Lemma. Just as in our preliminary detour at the beginning of this proof, whether or not they do will ultimately depend on the particular dynamics of the map $\tau$, acting on $I$. If, for example, $\tau$ is a cyclic permutation, then item v) is satisfied by the above choice of $U_j$ simply by taking $W_j^i = R_{\tau^i(\alpha)j}^{\delta}$ for each $j\in J$ and any $\alpha \in I$ (so that $k_j=|J|$). If, on the other hand, $\tau :I \rightarrow I$ is the identity, then there would have been no need to perform the pizza slice decomposition in the first place. But given that we have done it, the trapping regions sought are simply the $R_{ij}^{\delta}$. All we have to do is to is to enumerate them and call them $U_1, \ldots, U_N$ and take $W_j^1 = U_j$ (so that $N = |I \times J|$ and $k_j=1$ for every $j \in \{1, \ldots, N\}$). 

For the general case, we consider the periodic orbits $O_1, \ldots, O_s \subset I$ under $\tau$ and recall that, since every $i\in I$ falls into one of these after at most $|I|$ iterates, we can partition $I$ into $\tilde{O}_j = \bigcup_{k=1}^{|I|} \tau^{-k}(O_j)$, $j=1, \ldots, s$
The trapping regions that satisfy the full statement of the Shredding Lemma are
\begin{equation}
\tilde{U}_{(r,j)} = \bigcup_{i\in \tilde{O}_r} R_{ij}^{\delta}, \quad (r,j) \in \{1, \ldots, s\}\times J,
\end{equation}
conveniently relabeled as $U_1, \ldots, U_N$ (so that $N= s \cdot |J|$). Associated to a trapping region $\tilde{U}_{(r,j)}$, the cyclic sets of small diameter required by item v) of the Shredding Lemma can then be chosen by taking $W_{(r,j)}^i = R_{\tau^i(\alpha)j}^{\delta}$ for any $\alpha \in O_r$ (so that $k_j=|O_r|$).

\end{proof}

\begin{proof}[Proof of the Shredding Lemma for continuous mappings] 

We start by fixing some arbitrary $f \in C^0(M)$ and $\epsilon>0$. The aim of the proof is to describe how to find $g \in C^0(M)$ with $d(g, f)<\epsilon$ such that $g$ satisfies items i) to v) of the Shredding Lemma.

Let $\mathcal{R}=\{R_i\}_{i \in I}$ be a fine triangulation of $M$. Precisely how fine it should be is a question we postpone to the end of the proof. For the moment we content ourselves by requiring that $\dim \mathcal{R}$ be less than $\epsilon$. Subdivide each simplex $R_i \in \mathcal{R}$ into a union of subsimplices $R_i = \bigcup_{j \in J} R_{ij}$ so that 
\begin{equation}
m\left(\bigcup_{i \in I} R_{ij} \right) < \epsilon \quad \text{ for every } j \in J.
\end{equation}
Any subdivision will do. It does not have to be of pizza slice type, as it did in the case of homeomorphisms. In fact, when $\dim M =1$ the $R_i$ are intervals and the only subdivision possible is to write $R_i$ as a union of smaller intervals $R_{ij}$ that intersect at most on their end points.
For $\delta>0$, let $R_{ij}^{\delta} = \{x \in R_{ij}: d(x, \partial R_{ij}) > \delta \}$. We choose $\delta$ small enough so that  
\begin{equation}
m\left( \bigcup_{(i,j) \in I \times J} R_{ij}^{\delta} \right) > 1- \epsilon.
\end{equation}

For each $i\in I$ choose some $\tau(i) \in I$ such that $f(R_i) \cap R_{\tau(i)} \neq \emptyset$. Choose points $p_{ij}$ in the interior of $R_{ij}$. If necessary, we reduce $\delta$ so that $p_{ij} \in R_{ij}^{\delta}$ for every $(i,j) \in I \times J$. We shall define $g$ in such a way that $g=f$ on each $\partial R_{ij}$ and such that $f$ takes the constant value $p_{\tau(i)j}$ on the whole of $\overline{R_{ij}^{\delta}}$. That is,
\begin{align}
& g \vert \partial R_{ij} = f\vert \partial R_{ij} \quad \text{ and } \label{identity on border} \\
& g \vert \overline{R_{ij}^{\delta}} = p_{\tau(i)j}  \label{constant image} 
\end{align}
for every $ i\in I$ and $j \in J$.

 \begin{figure}[h]
 \center
 \includegraphics[width=4.0in]{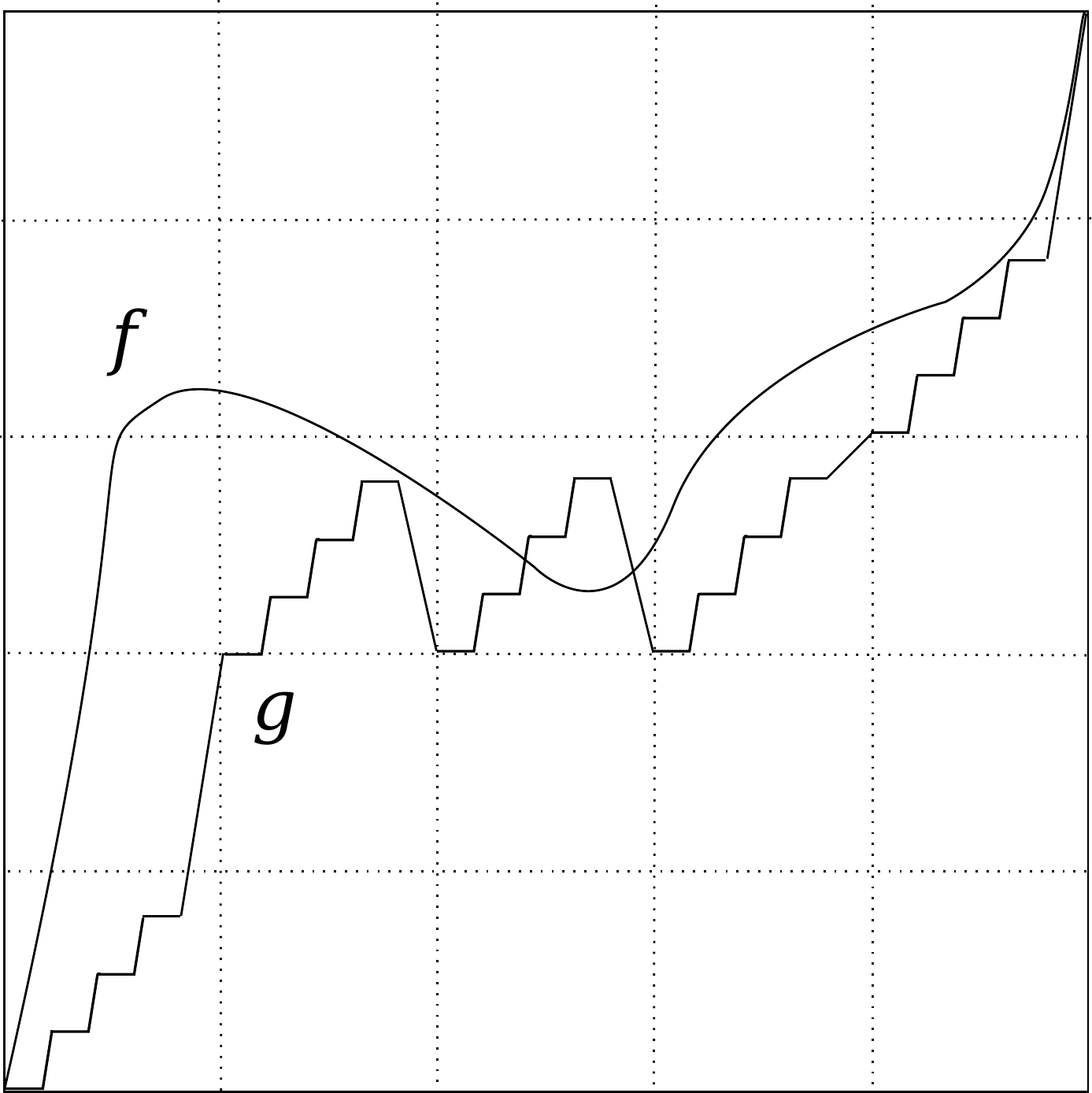}
 \caption{On $S^1$, the Shredding Lemma is obtained by perturbing the original map with a step function. In the above illustration, $S^1=R_1 \cup \ldots \cup R_5$ and $R_i=R_{i1}\cup \ldots \cup R_{i4}$ for $i=1, \ldots, 5$. The map $\tau$ has two periodic orbits, namely the fixed points $\{ 1 \}$ and $\{ 3 \}$ (if the $R_i$ are ordered from left to right). In particular our construction gives eight trapping regions in this case.}
 \end{figure}

Using the same notation as in the proof of the homeomorphism case, i.e. by considering the partition $\tilde{O}_1, \ldots, \tilde{O}_s$ of $I$, consisting of pre-images of the periodic orbits $O_1, \ldots, O_s$ under $\tau$, it shold now be clear that any $f:M \rightarrow M$ satisfying (\ref{identity on border}) and (\ref{constant image}) also satisfies items i) to v) of the Shredding Lemma with $U_j$ being a relabeling of the sets \begin{equation}
\tilde{U}_{(r,j)} = \bigcup_{i\in \tilde{O}_r} R_{ij}.
\end{equation}
Hence what remains to do is to make sure that $g$, as defined in (\ref{identity on border}) and (\ref{constant image}), can be extended continuously to $M$ in a way that $d(f,g) < \epsilon$.
By supposing that the triangulation $\mathcal{R}$ is sufficiently fine, we may assume that the charts $\xi_i$ extend to homeomorphisms $\psi_i: V_i \rightarrow \psi_i(V_i)\subset \mathbb{R}^n$ on open sets $V_i \supset R_i$ such that, for every $i \in I$, the closed convex hull of $f(R_i) \cup R_{\tau(i)}$ is contained in $V_{\tau(i)}$. More precisely, writing $i' = \tau(i)$, we require that $f(R_i) \subset V_{i'}$ and that the closed convex hull of $\psi_{i'} (f(R_i) \cup R_{i'})$ be contained in $\psi_{i'}(V_{i'})$. In the linear structure induced by $\psi_{i'}$ we may define $g$ on each $R_{ij}$, explicitly by fixing a continuous function $\varphi_{ij}:R_{ij} \rightarrow [0,1]$ satisfying $\varphi_{ij} \vert \partial R_{ij} \equiv 0$ and $ \varphi_{ij} \vert R_{ij}^{\delta} \equiv 1$ and define $g$ by
\begin{equation}
g(x) = (1-\varphi_{ij}(x)) \cdot f(x) + \varphi_{ij}(x) p_{i'j'}
\end{equation}
for every $x \in R_{ij}$. But this way of writing is of course only a shorthand for the chart dependent expression
\begin{equation}
 g (x)  =  \psi_{i'}^{-1} \left[ (1-\varphi_{ij}(x)) \cdot \psi_{i'}(f (x)) + \varphi_{ij}(x) \psi_{i'}(p_{i'j})\right] 
\end{equation}
for every $x \in R_{ij}$. Recall that we need to determine under what circumstances we have $d(f(x), g(x))< \epsilon$ for every $x \in M$. This would be guaranteed if 
\begin{equation}\label{smaller than epsilonprime}
\|\psi_{i'} \circ g(x)- \psi_{i'} \circ f(x)\|< \epsilon_1,
\end{equation} 
where $\epsilon_1>0$ is a number such that $d(\psi_{i'}^{-1}(\tilde{x}), \psi_{i'}^{-1}(\tilde{y}))<\epsilon$ whenever $\|\tilde{x}-\tilde{y}\|<\epsilon_1$, with $\| \cdot \|$ denoting the Euclidian norm in $\mathbb{R}^n$. But 
\begin{align}
\|\psi_{i'} \circ g(x)-\psi_{i'} \circ f(x)\| & = \varphi_{ij} \|\psi_{i'}(p_{i'j'})-\psi_{i'}(f(x))\| \\
& \leq \|\psi_{i'}(p_{i'j'})-\psi_{i'}(f(x))\|.
\end{align}
Hence (\ref{smaller than epsilonprime}) would be guaranteed if $d(p_{i'j'},f(x))<\epsilon_2$, where $\epsilon_2>0$ is a number such that $\|\psi_{i'}(x)-\psi_{i'}(y)\|< \epsilon_1$ whenever $d(x,y)< \epsilon_2$. Now, recall that both $p_{i'j'}$ and $f(x)$ belong to a set of the form $f(R_{i}) \cap R_{\tau(i)}$ where $f(R_{i}) \cap R_{\tau(i)} \neq \emptyset$. Hence we will always have $d(p_{i'j'},f(x))<\epsilon_2$ provided that the diameter of $\mathcal{R}$ is small enough. It is therefore essencial that we can take the diameter of $\mathcal{R}$ to be as small as we want \emph{without changing the charts} $\xi_i$, i.e. so that the same numbers $\epsilon_1$ and $\epsilon_2$ still have the properties we require of them for arbitrarily fine triangulations. But that is easily done simply by subdividing a given triangulation into a finer one, using the same charts for all the subsimplices of the finer triangulation. 

\end{proof}

\section{Proof of Theorem \ref{theoB}}\label{sectiontheoB}

We shall prove that the nonwandering sets of generic homeomorphisms in $\Hom(S^1)$ are zero Lebesgue measure Cantor sets of periodic points, and then deduce the desired dynamical consequences from this. We remark that Akin-Hurley-Kennedy have already proved (see Theorem $6.4$ on page $68$ of \cite{MR1980335}) that the nonwandering sets of generic homeomorphisms of arbitrary manifolds are Cantor sets, but we include a full proof (i) because on the circle the proof is simpler and shorter; (ii) for the sake of completeness; and (iii) because we use some of the vocabulary of the proof when we subsequently prove that these nonwandering sets have zero Lebesgue measure.

\begin{proposition} \label{prop.cantor}
The nonwandering sets of generic homeomorphisms in $\Hom(S^1)$ are zero Lebesgue measure Cantor sets of periodic points.
\end{proposition}

\begin{proof}
We deal only with the case of orientation preserving homeomorphisms, as the orientation reversing case may be obtained through minor modifications. Thus let $\Hom_+(S^1)$ be the set of orientation preserving homeomorphisms of the circle. The proposition is proved in a series of steps:

\begin{description}

\item[Step 1:] There is an open-and-dense subset $\cO$ of $\Hom_+(S^1)$ such that every $f \in \cO$ has rational rotation number; moreover given any $f \in \cO$, its rotation number is constant in a neighborhood of $f$.

First, by closing recurrent orbits it follows that there is a dense subset $\cD$ of $\Hom(S^1)$ such that every $f \in \cD$ has at least one periodic orbit (i.e., a rational rotation number). With an additional small $C^0$ perturbation one produces a topologically transversal periodic orbit $p$ (i.e., a periodic point of period $\pi$, say, such that that $f^{\pi}$ has a lift $F: \mathbf{R} \rightarrow \mathbf{R}$ satisfying $F(x)< F(p)=p < F(y)$ or $F(x)>F(p)=p > F(y)$ for some $x<p<y$). The existence (although not the uniqueness) of a transversal periodic orbit is a $C^0$-open condition by the intermediate value theorem. This implies that there is an open-and-dense subset $\cO$ of $\Hom(S^1)$ such that every $f \in \cO$ has rational rotation number. Moreover, the rotation number is constant in a neighbourhood of $f$ in $\cO$. We remark that it is a well-known fact that if the rotation number is rational then $\Omega(f) = Per(f)$.

\item[Step 2:] There is a residual subset $\cR_1$ of $\cO$ (and hence of $\Hom(S^1)$) such that the nonwandering set of every $f \in \cR_1$ has empty interior.

Fix a countable open basis $\{I_k\}$ of intervals of the circle. Given an interval $I_k$ of the basis and a homeomorphism $f \in \cO$ whose periodic orbits have a given period $\pi$, we can always perturb $f$ so that $f^{\pi}|_{I_k}$ does not coincide with the identity -- which means that the set $I_k$ does not consist of periodic points. Clearly this is a $C^0$-open condition. So there is an open-and-dense set $\cA_k$ of homeomorphisms $f$ such that $I_k \setminus Per(f)$ is (open and) nonempty. Taking the intersection we obtain a set 
\begin{equation}
\cR_1 \equiv \bigcap_{k \in \NN} \; \cA_k,
\end{equation}
residual in $\cO$, such that, by construction, if $f \in \cR_1$ then $Per(f)$ does not contain any interval.

\item[Step 3:] There is a residual subset $\cR_2$ of $\cO$ (and hence of $\Hom_+(S^1)$) such that the nonwandering set of every $f \in \cR_2$ is perfect.

Given $k, n \in \NN$ we shall prove the following: there is an open-and-dense subset $\cB_n^k$ of $\cO$ such that if $f \in \CB_n^k$ is such that $Per(f) \cap I_k \neq \emptyset$, then $\# (Per(f) \cap I_k) \geq n$. Indeed, given $f \in \cO$, then either (i) $Per(f) \cap I_k = \emptyset$ -- and this is a $C^0$-open condition by the upper-semicontinuous variation of the set $Per(f)$ with $f$ --; or (ii) $Per(f) \cap I_k \neq \emptyset$. In the second case, we perturb $f$ around (the preimage of) some periodic point in $Per(f) \cap I_k$ so as to ``unfold it'' into at least $n$ topologically transverse periodic points $p_1, \ldots, p_n \in I_k$. This condition is also $C^0$-open and it therefore yields the desired set $\cB_n^k$. Now take the (residual) intersection $\cR_2 \equiv \bigcap_{n, k \in \NN} \cB_n^k$. By construction, given $f \in \cR_2$ and any basic interval $I_k$ then either $Per(f) \cap I_k = \emptyset$ or $\#(Per(f) \cap I_k) = \infty$. This clearly implies that the set $
 Per(f)$ (which coincides with $\Omega(f)$) is perfect.

\end{description}

The homeomorphisms in $\cR_1 \cap \cR_2$ have Cantor nonwandering sets which consist of periodic orbits: their nonwandering sets are (i) of course compact, (ii) coincide with the set of periodic points by rationality of the rotation number, (iii) have empty interior by step 2, and (iv) are perfect by step 3. It remains to show that, generically, the set of periodic points has zero Lebesgue measure.

\begin{itemize}
\item Step 4: There is a residual subset $\cR_3$ of $\cO$ (and hence of $\Hom_+(S^1)$) such that the set of periodic points $Per(f)$ of every $f \in \cR_3$ has zero Lebesgue measure.

Given any homeomorphism $f \in \cO$, we can via a small $C^0$-perturbation smooth it into a diffeomorphism $f \in \cO$, which in turn can be perturbed into a Morse-Smale diffeomorphism, which has a finite (and hence zero Lebesgue-measure) set of periodic points. That is, there is a dense subset $\cD$ of $\cO$ which consists of homeomorphisms $f$ such that $m(Per(f)) = 0$. Now, given $\eps > 0$, by the upper semicontinuity of the map $f \to Per(f)$ there is an open neighborhood $\cU_{\eps}(f)$ of $f$ in $\cO$ such that if $g \in \cU_{\eps}(f)$ then $m(Per(g)) < \eps$. Now define $\cW_{\eps} \equiv \bigcup_{f \in \cD} \; \cU_{\eps}(f)$ and set $\cR_3 \equiv \bigcap_{n \in \NN} \; \cW_{\frac{1}{n}}$ to obtain the desired residual subset of $\cO$.

\end{itemize}

\end{proof}

We now deduce Theorem \ref{theoB} from Proposition \ref{prop.cantor}.

\begin{proof}[Proof of Theorem \ref{theoB}]

Again, we only deal explicity with the case of orientation preserving homeomorphisms, leaving the details of the orientation reversing case to the reader. Thus let $f$ be a generic homeomorphism in $\Hom_+(S^1)$ whose periodic points have period $\pi$. Then the (open and full-Lebesgue) set $S^1 \setminus Per(f)$ consists of a countable union of pairwise disjoint open intervals $I$ such that $f^{\pi}(I) = I$ and moreover the extremes of $I$ are two periodic points $p_1$ (on the left) and $p_2$ (on the right), which are necessarily extremal points of the Cantor set $Per(f)$.

Let $F$ be the lift of $f^{\pi}$ to the real line which has fixed points. Given an interval $I$ as above, there are two possible cases: either the graph of $F$ restricted to $I$ is below the identity, or else the graph of $F$ restricted to $I$ is above the identity.

In the first case, by dynamical monotonicity all of the points $x \in I$ converge in the future to the orbit of $p_1$: $d(f^{\pi k}(x), f^{\pi k}(p_1)) \to 0$; in the second case, the points $x$ of $I$ converge in the future to the orbit of $p_2$: $d(f^{\pi k}(x), f^{\pi k}(p_2)) \to 0$. This means that the periodic Dirac measure associated to the orbit of $p_1$ (in the first case) or of $p_2$ (in the second case) contains $I$ in its basin of attraction. In other words, Lebesgue-a.e. point of the circle belongs to the basin of attraction of the Dirac measure associated to an extremal point of the Cantor set $Per(f)$. This shows that $f$ is indeed countably wonderful, as claimed.
\end{proof}

\section{Proof of Theorem \ref{theoC}}

We recall the statement of Theorem \ref{theoC}: A generic continuous circle map, topologically conjugated to a linear expanding one, is wicked. `Generic' here means `generic in the induced $C^0$ topology on the set of all continuous maps conjugated to a linear expanding one'. In proving the theorem, we do not work directly with the maps themselves, but rather with the conjugating homeomorphisms. More precisely, what we actually prove is the following: Let $E_\ell$ denote the linear expanding circle map $x \mapsto \ell x$, for some integer $\ell$ with $|\ell| \geq 2$. Then for a generic circle homeomorphism $h$, the map $f = h^{-1}E_\ell h$ is wicked. The statement of Theorem \ref{theoC} then follows by the following proposition.

\begin{proposition} \label{locally homeomorphic}
\ 
\begin{enumerate}
\item The decomposition $CE(S^1) = \bigcup_{\vert \ell \vert \geq 2} CE_\ell(S^1)$ is a decomposition into isolated sets.
\item For each integer $\ell$ with $\vert \ell\vert \geq 2$, $CE_\ell(S^1)$ is locally homeomorphic to $\Hom_+(S^1)$. In fact, the map
\begin{align}
\Psi: \Hom_+(S^1) & \rightarrow CE_\ell(S^1) \\
h & \mapsto h^{-1} E_\ell h
\end{align} 
is a $\vert \ell-1 \vert $-to-one surjection, mapping  a neighbourhood of every $h$ in $\Hom_+(S^1)$ homeomorphically onto its image. In particular, $CE(S^1)$ is a Baire space.
\end{enumerate}
\end{proposition}

\begin{proof}
Suppose that  $f$ and $g$ are two elements of $CE_\ell(S^1)$ of degrees $\ell$ and $m$, say. 
Let $F$ and $G$ be lifts of $f$ and $g$ respectively. If $\ell \neq m$, then $deg(f-g) = \ell - m \neq 0$. It other words, $\{F(1)-G(1)\}-\{F(0)-G(0)\}$ is a nonzero integer. By the mean value theorem, there exists $0\leq x_0 \leq 1$ such that $F(x_0)-G(x_0) = n + 1/2$ for some integer $n$. Consequently $d_C^0(f,g) = 1/2$. This proves item $1$ of the theorem.

To prove item $2$ we fix some $\ell$ with $\vert \ell \vert \geq 2$ and write $E= E_\ell$ to simplify the notation a bit. We also fix some $h$ in $\Hom_+(S^1)$ and consider the map $f= h E h^{-1}$, which belongs to $CE_\ell(S^1)$. 
Since $f$ is conjugated to $E$, it must have exactly $\vert \ell-1 \vert$ fixed points, say $p_1, \ldots, p_{\vert \ell-1 \vert }$, and one of these must be mapped by $h$ to the poitn $0$ in $\mathbb{R} / \mathbb{Z}$. Once specified, the images, under $h$, of all other fixed points are also specified, since they must be mapped to the remaining fixed points of $E$ in a particular order. So are the images of all the fixed points of $f^2$, since they must be mapped into the set of fixed points of $E^2$ respecting a given order. By the same reasoning, the images of all periodic points are determined. It is thus clear by the density of periodic points that to specify which $p_i$ is mapped to the origin really determines $h$. It follows that, given $f \in CE_\ell(S^1)$, there are \emph{at most} $\vert \ell-1 \vert $ choices of $h \in \Hom_+(S^1)$ such that $f = h^{-1}E h$. On the other hand, composing $h$ on the left by a rotation $R$ whose angle is a multiple of $(\ell-1)^{-1
 }$, we obtain a new homeomorphism $h' = R h$ such that $(h')^{-1} E h' = f$. Hence there are \emph{at least} $\vert \ell-1 \vert$ choices of homeomorphisms that conjugate $f$ to $E$. We have therefore shown that there are \emph{precisely} $\vert \ell-1 \vert $ homeomorphisms conjugating a given $f$ to $E$, and that they all differ by left composition of a rigid rotation of angles that are multiples of ${\vert \ell-1 \vert}^{-1}$.

Given an arbitrary homeomorphism $h \in \Hom_+(S^1)$ we choose some neighbourhood $U$ of $h^{-1}(0)$ such that the pre-image, under $h$, of all other fixed points of $E$, do not intersect $\overline{U}$. Then the restriction of $\Psi$ to the set $\mathcal{U} = \{h^{-1}Eh: h^{-1}(0) \in U \}$ is injective. Continuity of $\Psi \vert \mathcal{U}$ is merely a matter of inspection. To see why its inverse is continuous, fix again an arbitrary element $h$ of $\Hom_+(S^1)$ and some neighbourhood $\mathcal{V}$ of $\Psi(h)$. Fix some large $k$ such that $\Psi(\tilde{h})\in \mathcal{V}$ whenever $\tilde{h}^{-1}(p) = h^{-1}(p)$ for every periodic point $p$ of period less than or equal to $k$. By labeling these points $p_1, \ldots, p_{\vert \ell^k-1 \vert}$, where $p_i = h^{-1}(i/\ell^k)$, and by choosing sufficiently small neighbourhoods $U_1, \ldots U_{\vert \ell^k-1 \vert}$ of them, we guarantee that the (open) set 
\[\mathcal{U}_0 = \{\tilde{h} \in \Hom_+(S^1): \tilde{h}^{-1}(i/\ell^k) \in U_i \} \]
is mapped by $\Psi$ into $\mathcal{V}$, proving that the inverse of $\Psi \vert \mathcal{U}$ is continuous.
\end{proof}

 In virtue of Proposition \ref{locally homeomorphic}, it is enough to prove the following: Fix an integer $\ell$ with $\vert \ell \vert \geq 2$ and let $E$ denote the linear expanding circle map of degree $\ell$. Then, given a generic orientation-preserving homeomorphism $h: S^1 \rightarrow S^1$, the map $f = h^{-1}E h$ is wicked, meaning that the sequence $\sum_{k=0}^{n-1}h_*^{-1} E_*^k h_* m$ is dense in $\M_f(S^1)$. But $h_*$ maps $\M_f$ homeomorphically onto $\M_E$, so it is indeed enough to prove that, for generic $h\in \Hom_+(S^1)$, the sequence $\sum_{k=0}^{n-1} E_*^k h_* m$ is dense in $\M_E$. That is what we are going to do throughout the remainder of this section.

We identify the circle with the interval $I^0 = [0, 1)$. We denote by $\mathcal{A}$ the alphabet $\{0, \ldots, \ell-1 \}$ and write $\mathcal{A}^p$ for the set $\mathcal{A} \times \ldots \times \mathcal{A}$ ($p$ times) of words of length $p$. If $\alpha \in \cA^p$ and $\beta \in \cA^q$ we denote by $\alpha \beta \in \cA^{p+q}$ their concatenation. That is, if $\alpha = 010$ and $\beta = 11$, then $\alpha \beta = 01011$. For each $p \in \mathbb{N}$ we partition $I^0$ into $\ell^p$ intervals $\{I_{\alpha}^p : \alpha \in \mathcal{A}^p \}$, where
\[I_{\alpha}^p = \left[ \frac{\alpha}{\ell^p}, \frac{\alpha +1}{\ell^p} \right), \]
treating $\alpha$ as a natural number expressed in base $\ell$. Thus if $\ell = 2$ and $p=3$ we have $I_{000}^3 = [0, \frac{1}{8})$, $I_{001}^3 = [\frac{1}{8}, \frac{2}{8})$ and $I_{111}^3 = [\frac{7}{8}, 1)$.

Every $h \in \Hom_+(S^1)$ gives rise to a sequence of partitions $\cJ^p = \{ J_{\alpha}^p : \alpha \in \cA^p \}$, $p \in \mathbb{N}$, given by $J_{\alpha}^p = h^{-1}(I_{\alpha}^p)$. The sequence $\cJ^p$ (just like $\cI^p$) is \emph{consistent} in the following sense: for every $p, q \in \mathbb{N}$, and every $\alpha \in \cA^p$, we have
\begin{equation}\label{consistent}
 J_{\alpha}^p = \bigcup_{\beta \in \cA^q} J_{\alpha \beta}^{p+q}.
\end{equation}
Note that $h_*m(I_{\alpha}^p) = m (J_{\alpha}^p)$. Moreover, for $q \geq 0$ we have
\begin{equation}\label{measure of push-forward}
E_*^q h_*(I_{\alpha}^p) = m\left( \bigcup_{\beta \in \cA^q} J_{\beta \alpha}^{q+p} \right)
= \sum_{\beta \in \cA^q} m( J_{\beta \alpha}^{q+p}).
\end{equation}

Conversely, given any finite sequence of partitions $\cJ^1, \ldots, \cJ^q$ into, respectively, $\ell, \ldots, \ell^q$ intevals, consistent in the sense of (\ref{consistent}), there exists a homeomorphism $h:S^1 \rightarrow S^1$ (e.g., a piecewise linear one) such that $h(J_{\alpha}^j) = I_{\alpha}^j$ for every $1 \leq j \leq q$ and $\alpha \in \cA^j$.

To prove Theorem \ref{theoC}, we fix a countable base $\{V_i\}_{i \in \mathbb{N}}$ of the weak*-topology on $\cM(S^1)$ where each $V_i$ is of the form
\begin{equation}
 V_i = \left\{ \mu\in \M(S^1): \left\vert \int \varphi_i^j \ d\mu - \int \varphi_i^j \ d\nu_i \right\vert < \epsilon_i^j \
\forall 1\leq j \leq k_i \right\}
\end{equation}
for some $\nu_i \in \M(S^1)$, $k_i \in \mathbb{N}$, continuous $\varphi_i^j:S^1 \rightarrow \mathbb{R}$, numbers $\epsilon_i^j >0$, and consider the (open) sets 
\begin{equation}
\cU_{i,n}=\left\{h \in \Hom_+(S^1): \frac{1}{n} \sum_{k=0}^{n-1} E_*^k h_*m \in V_i \right\}.
\end{equation} 
\begin{lemma}\label{Uin is dense}
 Suppose $V_i \cap \cM_E(S^1) \neq \emptyset$ and let $m$ be any integer. Then $\bigcup_{n\geq m} \cU_{i,n}$ is dense in $\Hom_+(S^1)$.
\end{lemma}
Once the above claim is proved, the proof of Theorem \ref{theoC} follows by observing that
$\frac{1}{n} \sum_{k=0}^{n-1} E_*^k h_*m$ accumulates on the whole of $\cM_E(S^1)$ if and only if
\begin{equation}
 h \in \bigcap_{ \substack{i \in \mathbb{N} \text{ such that } \\ V_i \cap \cM_E(S^1) \neq \emptyset}} \ \bigcap_{m \geq 0} \  \bigcup_{n \geq m} \cU_{i,n}.
\end{equation}

\begin{proof}[Proof of Lemma \ref{Uin is dense}]
Fix $h \in \Hom_+(S^1)$, $V_i$ such that $V_i \cap \cM_E (S^1) \neq \emptyset$ and $\epsilon>0$. The goal is to prove that if $n_0 \geq 1$ is sufficiently large, then for every $n > n_0$ there is some $h' \in \Hom_+(S^1)$ with $d(h', h) <\epsilon$, such that
\begin{equation} \label{inside V}
 E_*^k {h'}_* m \in V_i \quad \forall n_0 \leq k \leq n-1.
\end{equation}
For then
\begin{equation}
 \frac{1}{n}\sum_{k=0}^{n-1} E_*^k {h'}_*m
= \frac{1}{n}\sum_{k=0}^{n_0-1} E_*^k {h'}_*m + \frac{1}{n}\sum_{k=n_0}^{n-1} E_*^k {h'}_*m \in V_i
\end{equation}
provided that $n$ is sufficiently large in comparison to $n_0$. To this end, let $n_0$ be any integer satisfying
$\ell^{-n_0} < \epsilon$. Next pick some $\mu \in V_i \cap \cM_E(S^1)$ and choose $p$ large enough so that $\nu \in V_i$ whenever $\nu$ is a measure satisfying
\begin{equation}
 \nu(I_{\alpha}^p) = \mu(I_{\alpha}^p) \quad \forall \alpha \in \cA^p.
\end{equation}

We define a consistent family of partitions $\{\cJ^k \}_{k=1}^{n-1}$ in the following manner: For $1\leq k \leq n_0$ and $\alpha \in \cA^k$ let $J_{\alpha}^k = h^{-1}(I_{\alpha}^k)$. For $n_0 < k \leq n-1$ and $\alpha \in \cA^k$ write $\alpha = \beta \gamma$ with $\beta \in \cA^{n_0}$ and $\gamma \in \cA^{k-n_0}$. Define $\cJ^k$ in such a way that
\begin{equation}\label{defn of J}
 m(J_{\beta \gamma}^k ) = m(J_{\beta}^{n_0}) \mu(I_{\gamma}^{k-{n_0}}).
\end{equation}
Let $h':S^1 \rightarrow S^1$ be a homeomorphism such that $h'(J_{\alpha}^k) = I_{\alpha}^k$ for every $\alpha \in \cA^k$, $1\leq k \leq n-1$. Then $d(h', h)< \epsilon$ since $h$ and $h'$ agree on each $J_{\alpha}^{n_0}$. Moreover (\ref{inside V}) holds for such a choice of $h'$. Indeed, when $n_0 \leq k\leq n-1$ we may write $\beta \in \cA^k$ as $\omega \tau \in \cA^{n_0} \times \cA^{k-n_0}$. Thus combining (\ref{measure of push-forward}) and (\ref{defn of J}) we have
\begin{align}
 E_*^k h_* m(I_{\alpha}^p) & = \sum_{\omega \in \cA^{n_0}} \sum_{\tau \in \cA^{k-n_0}}
m( J_{\omega \tau \alpha}^{k+p})\\ & = \sum_{\omega \in \cA^{n_0}} \sum_{\tau \in \cA^{k-n_0}}
m(J_{\omega}^{n_0}) \mu( I_{\tau \alpha}^{k-n_0 + p}) \\ & = \sum_{\tau \in \cA^{k-n_0}}
\mu( I_{\tau \alpha}^{k-n_0 + p})  = \mu ( E^{-(k-n_0)}(I_{\alpha}^p)) = \mu(I_{\alpha}^p).
\end{align}
By our choice of $p$ this implies that $E_*^k {h'}_* m \in V_i$ as required.
\end{proof}

\bibliographystyle{alphanum}
\bibliography{bibdoflavio}{}

{\bf Fl\'avio Abdenur} (flavio.abdenur@gmail.com) Departamento de Matem\'atica, PUC-Rio de Janeiro, 22453-900, Rio de Janeiro - RJ, Brazil.

\vskip 12pt

{\bf Martin Andersson} (martin@mat.uff.br) Universidade Federal Fluminense (GMA), 24020-140, Niter\'oi - RJ, Brazil. 

\end{document}